\documentclass[11pt]{amsart}
\usepackage[utf8]{inputenc}
\usepackage{amsmath,amsthm,amssymb,amsfonts}
\usepackage[T1]{fontenc}
\usepackage{textcomp}

\usepackage{graphicx}
\usepackage{color}

\newtheorem{thm}{Theorem}[section]

\newtheorem{prop}[thm]{Proposition}

\newtheorem{lem}[thm]{Lemma}
\theoremstyle{definition}

\newtheorem{remark}[thm]{Remark}

\def\sep{{ \ \  }}

\def\seg{{\ \ \ \  \ \  \ \  }}

\def\N{{\mathbb N}}

\def\C{{\mathbb C}}

\def\seg{ \ \ \ \ \ \ }


\title[Continuity of the composition operation] {On the continuity of the composition operation on spaces of holomorphic mappings}

\author[M.D. Acosta]{Mar{\'\i}a D. Acosta }
   \address{Departamento de An{\'a}lisis
Matem{\'a}tico\\ Universidad de Granada\\ 18071 Granada, Spain}
\email{dacosta@ugr.es}

\author[P. Galindo]{Pablo Galindo}
    \address{Departamento de An{\'a}lisis
Matem{\'a}tico\\ Universidad de Valencia\\ 46100 Burjasot,
Valencia, Spain} \email{galindo@uv.es}

   \author[L.A. Moraes]{Luiza A. Moraes}
\address{Instituto de Matem\'atica \\
 Universidade Federal do Rio de Janeiro \\ CP 68530 - CEP 21945-970 Rio de
Janeiro, RJ \\  Brazil}
\email{luiza@im.ufrj.br}

\thanks{The  first  author was  supported  by Junta de Andaluc\'{\i}a grant  FQM--185  and also by Spanish MINECO/FEDER grant  MTM2015-65020-P. The second and third authors   were partially supported  by MTM 2014-53241P-MINECO (Spain)}

\begin{document}


 \maketitle

\today

    \begin{abstract}
We discuss the continuity of the \emph{composition} on several spaces of holomorphic mappings on open subsets of a complex Banach space.  On the Fr\'{e}chet space of the entire mappings  that are bounded on bounded sets  the  composition  turns to be even holomorphic. In such space we consider linear subspaces closed under left and right composition. We discuss  the relationship of such subspaces with  ideals of operators and give several examples of them.	
We also provide natural examples of spaces of holomorphic mappings where the composition is not continuous.
\end{abstract}

\section{Introduction and Background}
\label{intro}

 Throughout this paper $E$ is a complex Banach space and $U$ is an open subset of $E.$ We denote by $B_E$ the open unit ball of $E$
    and by $\mathcal{H}(U,E)$ the space of the holomorphic mappings  $f:U
\rightarrow E$, i. e., Fr\'{e}chet differentiable at every point in $U$. We denote by $\mathcal{H}(U,U)$ the set of the elements  $f \in
\mathcal{H}(U,E)$ such that $f(U) \subset U$.

 As usual, we denote by  $\tau_0$ the compact-open topology on $\mathcal{H}(U,E).$ The family  of seminorms $p_K$ given by
$$
p_K (f)= \max \{ \Vert f(x) \Vert : x \in K\} \seg (f \in \mathcal{H}(U,E)),
$$
where $K $ is any compact subset of $U$ generates  $\tau_0$.

Recall that $B \subset U $ is {\emph U-bounded} if $B$ is bounded and $d(B,\partial U)>0$. As usual,  $\mathcal{H}_b(U,E)$ denotes the space of holomorphic mappings from $U$ into $E$ which are bounded on the $U$-bounded subsets of $U$ endowed with the topology $\tau_b$ of uniform convergence on  $U$-bounded subsets of $U.$

This paper is concerned with the continuity of the \textsl{composition} operation on spaces
 of holomorphic mappings on open subsets of a complex Banach space. We mainly focus on  the Fr\'{e}chet space of entire mappings of bounded type where the composition operation turns to be holomorphic.
We  also prove that the composition operation is continuous  on $\big(\mathcal{H}(B_E,B_E),\tau_0\big).$
When working in  $\big(\mathcal{H}(B_E,B_E),\tau_0\big)$ one frequently
 faces with the failure of Montel theorem in the infinite-dimensional case.
 This motivated us to consider another natural  topology defined for spaces of holomorphic mappings acting on open subsets of a dual space  and related to Montel theorem.
It might have been useful  that $\mathcal{H}(B_{E^*},B_{E^*})$ endowed with such topology and the  composition operation would have been a topological semigroup. However we   prove that the composition is not continuous for such topology  (see  Theorem \ref{cont-tau-0-*}).

Borrowing the notion of ideal in the algebra $\mathcal{L}(E)$ of  bounded linear operators on $E$  we study its analogue in $\mathcal{H}_b(E,E)$ by considering linear subspaces closed under left and right composition, that we  call  ideals as well. Among them, $\mathcal{H}_k(E,E),$ the subspace of holomorphic mappings that map bounded sets into relatively compact ones  that in turn is, whenever $E$ has the approximation property, the smallest proper closed ideal  which contains some non-constant mapping. It is well known that the  subspace $\mathcal{H}_{wu}(E,E)\subset \mathcal{H}_b(E,E)$ of the holomorphic mappings that are weakly uniformly continuous on  bounded subsets of $E$  lies inside $\mathcal{H}_k(E,E)$ and we prove that both subspaces coincide if and only if $\mathcal{H}_{wu}(E,E)$ is an ideal in $\mathcal{H}_b(E,E)$. This fact is equivalent to the absence of copies of $\ell_1$ in $E.$

We provide techniques that allow to pass from ideals in $\mathcal{L}(E)$ to ideals in $\mathcal{H}_b(E,E).$
Some results in the holomorphic setting are quite different from the ones in the linear setting. First of all, $\{0\}$ is not an ideal in $\mathcal{H}_b(E,E)$ and the smallest ideal in this case is the ideal of the constant mappings. For  $E=c_0 \text{ or } \ell_p  \;\;(1 \leq p < \infty)$ it is well known that   $\mathcal{K}(E)$, the subspace  of the compact operators on $E$, and  $\{0\}$  are the only  closed proper ideals in  $\mathcal{L}(E).$ Remarkably, for these spaces,
the set of proper closed ideals  in $\mathcal{H}_b(E,E)$ different of the trivial ideal of
 constant mappings contains a smallest and a largest element which are different.


We refer to \cite{Mu} for the necessary background on infinite dimensional holomorphy.
\section{Positive results}
\begin{thm}
\label{cont}
 Let $E$ be a  Banach space. The mapping
$$
(f,g) \in (\mathcal{H}(B_E,B_E), \tau_0) \times (\mathcal{H}(B_E,B_E),\tau_0) \mapsto f \circ g \in
(\mathcal{H}(B_E,B_E), \tau_0)
$$
is continuous.
\end{thm}

\begin{proof}

First we notice that it is easy to check that $\mathcal{H}(B_E,B_E)$ is a bounded subset of  the space
$(\mathcal{H}(B_E,E),\tau_0)$ and consequently $\mathcal{F}
=\mathcal{H}(B_E,B_E)$ is equicontinuous by  \cite[Proposition 9.15]{Mu}. This means that for  each  $\varepsilon > 0$ and $a \in B_E$
there is $\delta_a > 0$ such that $B(a,\delta_a) \subset B_E$ and
\begin{equation}
\label{equicont--1}
\Vert h(u) - h(a) \Vert < \frac{\varepsilon}{3}\;\;\; \text{for all} \;\; \;u \in
B(a,\delta_a)\;\;\;\text{and for all}\;\;\;h \in \mathcal{F} = \mathcal{H}(B_E,B_E).
\end{equation}

Moreover, we show that $\mathcal{F} = \mathcal{H}(B_E,B_E)$ is uniformly equicontinuous on compact subsets
of $B_E.$ Indeed, let $K \subset B_E$ be a  (nonempty) compact set, take  $\varepsilon > 0$  and  assume that
for each $a \in K$, $\delta _a$  is a positive real number satisfying  \eqref{equicont--1}. Since $K$ is
compact, there is a  (nonempty) finite subset $F \subset K$ such that
\newline

\begin{equation}
\label{unionK}
 K \subset \displaystyle{ \bigcup_{a \in F} \Bigl(  a+\frac{\delta _a}{3} B_E} \Bigr).
 \end{equation}

  Take
$\delta = \frac{1}{3} \min \{ \delta _a : a \in F\} > 0$. Choose any elements $x,y \in K$ such that $\Vert x
- y \Vert < \delta $. By \eqref{unionK} there are elements $a, b \in F$ such that $\Vert x-a \Vert <
\frac{\delta _a}{3}$ and  $\Vert y-b \Vert < \frac{\delta _b}{3}$. Hence $\Vert a-b \Vert < \max\{
\delta_a,\delta _b\}$.  So, by using \eqref{equicont--1} we obtain  that for every $x,y \in K$ such that
$\Vert x - y \Vert < \delta $ we have
\begin{eqnarray*}
\Vert h(x)- h(y) \Vert &\le&  \Vert h(x)- h(a) \Vert +   \Vert h(a)- h(b) \Vert +  \Vert h(b)- h(y) \Vert\\
&\le& \frac{\varepsilon}{3} + \frac{\varepsilon}{3} + \frac{\varepsilon}{3}= \varepsilon,
\end{eqnarray*}
for all $h \in
\mathcal{F} .$
This completes the proof of the uniform equicontinuity of $\mathcal{F}$ on compact subsets of $B_E.$

The next step is to prove:

 (A) For every net $(f_\alpha)_{\alpha \in \Lambda}$ contained in  $\mathcal{F},$
pointwise convergence is equivalent to uniform convergence on compact sets of the domain $B_E.$

Assume  that $(f_\alpha)_{\alpha \in \Lambda}$ is a net in $\mathcal{F}$ that converges pointwise to an
element $f \in \mathcal{F}$.  Let $K \subset B_E$ be a compact set and take  $\varepsilon > 0$. In view of
the uniform equicontinuity of $\mathcal{F}$ on compact subsets of $B_E$ we know that there exists   $\delta  > 0$ such that
\begin{equation}
\label{equicont2}
\Vert h(x) - h(y) \Vert < \frac{\varepsilon}{3},\;\; \text{for all} \;\;
h \in \mathcal{F} \;\;\text{whenever}\;\; x,y \in K \;\; \text{satisfy}\;\; \Vert x-y \Vert < \delta .
\end{equation}
As $K$ is compact, we can assume that $  K \subset  F_1 + \delta B_E$, for some finite set $F_1 \subset K$.
Since $F_1$ is finite and the net $(f_\alpha)_{\alpha \in \Lambda}$ converges pointwise to $f$, there is
$\alpha_0 \in \Lambda$ satisfying
\begin{equation}
\label{cong-point-1}
\Vert f_\alpha (a) - f(a) \Vert < \frac{\varepsilon}{3}, \sep \text{for all} \;\; a \in F_1 \;\; \text{and
for all} \;\; \alpha \ge \alpha_0.
\end{equation}
Hence, for any $x \in K$ we can find  $a \in F_1$ such that $\Vert x-a \Vert < \delta $. In view of
\eqref{equicont2}  and \eqref{cong-point-1},  for every   $\alpha \ge \alpha_0$, we obtain
\begin{eqnarray*}
\Vert f_\alpha(x) - f(x) \Vert &\le&  \Vert f_\alpha(x) - f_\alpha(a) \Vert +  \Vert f_\alpha(a) - f(a) \Vert
+ \Vert f(a) - f(x)
\Vert \\
&\le&
\frac{\varepsilon}{3} +  \frac{\varepsilon}{3}  + \frac{\varepsilon}{3}  =\varepsilon.
\end{eqnarray*}

We remark that $\alpha_0$ does not depend on the choice of $x \in K$  and so we have just proved that
$\bigl(f_\alpha  \bigr) \stackrel{\tau_0}{\longrightarrow } f$.

Finally, we are  going to show that for any net $((f_\alpha, g_\alpha))_{\alpha \in \Lambda} \in \mathcal{F}
\times \mathcal{F}$ the facts that $\bigl( f_\alpha \bigr)  \stackrel{\tau_0}{\rightarrow} f \in \mathcal{F}$ and $\bigl( g_\alpha  \bigr)
\stackrel{\tau_0}{\rightarrow} g \in \mathcal{F}$ always imply $\bigl( f_\alpha \circ g_\alpha  \bigr)
\stackrel{\tau_0}{\rightarrow} f \circ g.$

By (A), it is enough to prove that $(g_\alpha \circ f_\alpha)_{\alpha \in \Lambda}$
converges pointwise to $g \circ f$ on $B_E$. Let us take  $x \in B_E$ and  $\varepsilon >0$. In view of
\eqref{equicont--1} there is $\delta   > 0$ such that  $B(f(x) ,\delta ) \subset B_E$ and
\begin{equation}
\label{equicont-fx}
\Vert h(u) - h(f(x)) \Vert < \frac{\varepsilon}{2} \sep \text{for all}\;\; u \in B(f(x),\delta ) \;\;
\text{and for all} \;\; h \in \mathcal{F}.
\end{equation}
Since  $ \bigl(f_\alpha  \bigr)  \stackrel{\tau_0}{\to} f$  we have that   $ \bigl( f_\alpha(x)  \bigr)  \to f(x)$  and so there exists
$\alpha_1 \in \Lambda $ such that for every $\alpha \ge \alpha_1$ it is satisfied that $\Vert f_\alpha(x) -
f(x) \Vert < \delta$. In view of \eqref{equicont-fx} we deduce that
\begin{equation}
\label{g-alfa-f-alfa-x}
\Vert g_\alpha(f_\alpha(x)) - g_\alpha (f(x)) \Vert < \frac{\varepsilon}{2} \;\; \text{for all} \;\; \alpha
\ge \alpha_1.
\end{equation}

As  $\bigl(g_\alpha \bigr)  \stackrel{\tau_0}{\to} g$,   we know that   $ \bigl(g_\alpha( f(x)) \bigr)  \to g(f(x))$  and so there exists
$\alpha_2 \ge \alpha _1$  such that for every $\alpha \ge \alpha_2$  we have  $\Vert g_\alpha( f(x)) -
g(f(x)) \Vert < \frac{\varepsilon}{2}$. By  using also  \eqref{g-alfa-f-alfa-x},  for any $ \alpha \ge \alpha
_2 $  we obtain that
$$
\Vert g_\alpha(f_\alpha(x)) -  g (f(x)) \Vert \le
 \Vert g_\alpha(f_\alpha(x)) -  g_\alpha  (f(x)) \Vert + \Vert g_\alpha(f (x)) -  g (f(x)) \Vert <
 \frac{\varepsilon}{2} +   \frac{\varepsilon}{2}  = \varepsilon.
$$
We just proved that $(g_\alpha \circ f_\alpha)_{\alpha \in \Lambda}$ converges pointwise to $g \circ f$. As a consequence,
$(g_\alpha \circ f_\alpha)_{\alpha \in \Lambda}$ converges  to $g \circ f$ in the compact-open topology.
This completes the proof of the theorem.
\end{proof}

\begin{remark}
The proof of Theorem \ref{cont} contains the proofs of the following two facts that are probably known: Every
equicontinuous family of mappings is equiuniformly continuous on compact sets and pointwise convergence is
equivalent to uniform convergence on compact subsets of the domain for every net contained in an
equicontinuous family of mappings. We could not find these results proved in the literature and so we decided
to include the proofs for the sake of completeness.
\end{remark}
\begin{remark}
As a consequence of Theorem \ref{cont} we have that $(\mathcal{H}(B_E,B_E), \tau_0)$ endowed with the
composition is a topological semigroup. Moreover, it is clear that Theorem \ref{cont} remains true if we take $(\mathcal{H}(U,U), \tau_0)$ (where $U$ is a bounded open subset  of $E$) instead of $(\mathcal{H}(B_E,B_E), \tau_0).$
\end{remark}


It is natural to ask if Theorem \ref{cont} holds for the whole space $E.$  The negative answer follows from the following result that J. Bonet kindly pointed  out to us and that we thankfully acknowledge.
\begin{prop}[Bonet]
 Let $E$ be  an infinite dimensional real or complex  Banach space. The mapping
$$
(T,S) \in (\mathcal{L}(E), \tau_0) \times \big(\mathcal{L}(E), \tau_0\big) \mapsto T \circ S \in
\big(\mathcal{L}(E), \tau_0\big)
$$
is not continuous.
\end{prop}
\begin{proof}
First fix $x_0 \in E$ and $x_0^*  \in E^\ast$ such that $x_0^*(x_0)=1.$ Denote by $\tau_c$ the topology on $E^*$ of the uniform convergence on compact subsets of $E.$ It is  immediate  that the mappings
$$
(x^* ,x) \in (E^\ast, \tau_c) \times E \mapsto x^*  \otimes x_0 \in \big(\mathcal{L}(E), \tau_0\big)
$$
and
$$
(x^* ,x) \in (E^\ast, \tau_c) \times E \mapsto x_0^* \otimes x \in \big(\mathcal{L}(E), \tau_0\big)
$$
are both continuous. Hence the mapping
$B:(E^\ast, \tau_c) \times E  \rightarrow \big(\mathcal{L}(E), \tau_0\big) \times \big(\mathcal{L}(E), \tau_0\big)$
defined by $B(x^*,x):=(x^*  \otimes x_0 , x_0^* \otimes x )$ is also continuous.
Moreover, the scalar-valued mapping
$\Phi:\big(\mathcal{L}(E), \tau_0\big) \rightarrow \mathbb{K}$ defined by $\Phi(T)= x_0^* \big(T(x_0)\big)$ is clearly continuous.

Assume that  the composition mapping
$$
(T,S) \in \big(\mathcal{L}(E), \tau_0\big) \times \big(\mathcal{L}(E), \tau_0\big) \stackrel{C}\mapsto T \circ S \in
\big(\mathcal{L}(E), \tau_0\big)
$$
is continuous.  Then the mapping $\Phi \circ C \circ B$  is  continuous as well, and for each  $x^* \in E^*$ and $x \in E$ one has that
$$
(\Phi \circ C \circ B) (x^* ,x) =
(\Phi \circ C) ( x^* \otimes x_0 ,  x_0 ^* \otimes x ) = \Phi \left(( x^*  \otimes x_0) \circ (x_0 ^* \otimes x) \right)
$$
$$
= x_0 ^*\Big( ( x^*  \otimes x_0) \circ (x_0^* \otimes x)(x_0)\Big) =
 x_0 ^* \Big(( x^*  \otimes x_0)(x)\Big)  =  x_0 ^*  \big( x^* (x) \cdot x_0\big)  = x^* (x  ).
 $$
 The continuity of $\Phi \circ C \circ B$ implies the existence of a compact $K \subset E$ and $r > 0$ such that
 $$\text{ if } x \in r \overline{B}_E \text{ and } x^* \in E^* \text{ with }  \sup \{ \vert x^* (y) \vert : y \in K \} \le 1,    \text{ then }   \sep  \vert x^* (x)  \vert \le 1.
 $$
 \noindent
 That is,  $\vert x^* (x)\vert \le  \sup \{ \vert x^* (y) \vert : y \in K \} $  for all $x \in  r\overline{B}_E $ and $x^* \in E^*.$  From where it follows by Hahn-Banach theorem that $r\overline{B}_E$ lies inside $\overline{\text{aco}(K)},$ the closure of the absolutely convex hull of $K.$ Since $\overline{\text{aco}(K)}$ is compact, we have that $ \overline {B}_E $ is compact. This contradicts the hypothesis that $\dim E = \infty.$
\end{proof}

Let $\mathcal{H}_b(U,U)$ denote the topological subspace of $\mathcal{H}_b(U,E)$ of  holomorphic mappings that map $U$ into $U.$
\begin{prop}\label{cont-ac}
Let $U\subset E$ be a balanced open set. The composition mapping
$$
(f,g) \in \mathcal{H}_b(U,U) \times \mathcal{H}_b(U,U) \mapsto f \circ g \in
\mathcal{H}_b(U,U)
$$
is continuous.
\end{prop}
\begin{proof}
We begin by recalling that holomorphic mappings of bounded type are uniformly con\-ti\-nu\-ous on $U$-bounded sets as they may be approximated there by polynomials.

Let $((f_\alpha,g_\alpha))_{\alpha \in \Lambda}  \subset  \mathcal{H}_b(U,U) \times \mathcal{H}_b(U,U)$ be a net converging to $(f,g)$  and $B$ a $U$-bounded set. Since $g(B)$ is $U$-bounded, there is $\rho>0$ such that $C:=g(B)+\rho B_E \subset U$ and $C$  is still $U$-bounded.

Fix $\varepsilon >0,$ that we may assume (and do) $\varepsilon <\rho.$ The uniform continuity of $f$ on $C$ leads us to $\delta>0,\;\delta<\rho,$ such that $\|f(u)-f(v)\|\le \varepsilon $ if $u,v\in C$ and $\|u-v\|<\delta.$

Next we find $\alpha_0$ such that $\|g_\alpha(x)-g(x)\|\le \delta$ for $\alpha\ge\alpha_0$ and $x\in B.$ Thus, $g_\alpha(x) \text{  and } g(x)\in C$  for $\alpha\ge\alpha_0$ and $x\in B.$ Now, we pick $\alpha_1\ge \alpha_0$ such that $\|f_\alpha(y)-f(y)\|\le \varepsilon $ for $\alpha\ge\alpha_1$ and $y\in C.$ Therefore,  if $x\in B$  and  $\alpha \ge \alpha_1,$ we have
$$
\|f_\alpha(g_\alpha(x))-f(g(x))\|\le \|f_\alpha(g_\alpha(x))-f(g_\alpha(x))\|+\|f(g_\alpha(x))-f(g(x))\|\le \varepsilon +\varepsilon.$$\end{proof}
It is well known that  on any infinite dimensional Banach space $E$ the topologies $\tau_b$ and   $\tau _0$ do not coincide on
$\mathcal{H}_b(B_E, B_E)$.
\begin{thm}
	\label{Hb comp}
 The composition mapping $$
(f,g) \in \mathcal{H}_b(E,E) \times \mathcal{H}_b(E,E) \mapsto f \circ g \in \mathcal{H}_b(E,E)$$
is holomorphic.\end{thm}\begin{proof}
Since  $\mathcal{H}_b(E,E)$ is a Fr\'{e}chet space, it suffices to show that the composition mapping is separately holomorphic. For fixed $g \in \mathcal{H}_b(E,E),$ the mapping $f\mapsto f\circ g$ is clearly linear and, by Proposition \ref{cont-ac}, it is continuous, hence it is holomorphic.

Fix now $f \in \mathcal{H}_b(E,E).$ By Proposition \ref{cont-ac} the mapping  $g\mapsto f\circ g$ is continuous. So,  we have to show that it is a G-holomorphic mapping. Take any $g_1, g_2 \in  \mathcal{H}_b(E,E).$ We claim that the mapping
 \begin{equation} \label{sephol}\lambda \in \mathbb{C} \mapsto f \circ (g_1 + \lambda g_2)  \in \mathcal{H}_b(E,E)\end{equation} is analytic.    If the Taylor series of $f$ at $0$ is $f=\sum_m P_m ,$ we have $f \circ (g_1 + \lambda g_2)=\sum_m P_m\circ\big(g_1 + \lambda g_2\big) .$ For every $m\in\mathbb{N},$ the mapping $\lambda \in \mathbb{C} \mapsto P_m\circ\big(g_1 + \lambda g_2\big)$ is a polynomial with values in  $\mathcal{H}_b(E,E).$ So the claim will hold if the series $\sum P_m\circ\big(g_1 + \lambda g_2\big)$ converges in $\mathcal{H}_b(E,E)$ uniformly on compact sets in $\mathbb{C}.$ So, let $K\subset \mathbb{C}$ be a compact set and $B\subset E$ a bounded set. Since $g_1(B)$ and $g_2(B)$ are bounded sets in $E,$ also the balanced hull $L$ of the set $\{g_1(B)+\lambda g_2(B):\lambda\in K\}$ is a bounded set in $E.$ Hence by Cauchy inequalities, $\sup_{u\in L}\|P_m(u)\|\leq \frac{1}{2^m}\sup_{u\in 2L}\|f(u)\|.$ Thus $$\sup_{\lambda\in K}\sup_{u\in B}\|P_m\circ\big(g_1(u) + \lambda g_2(u)\big)\|\leq\frac{1}{2^m}\sup_{u\in 2L}\|f(u)\|,$$ that shows the uniform convergence of $\sum P_m\circ\big(g_1 + \lambda g_2\big)$ on $K$ in $\mathcal{H}_b(E,E).$

 Therefore, the mapping in (\ref{sephol}) is analytic and, so $g\mapsto f\circ g$ is $G$-holomorphic as wanted.
 \end{proof}

\section{Ideals in the space $\mathcal{H}_b(E,E)$}
We note that, endowed with the product defined by the composition $(f,g) \mapsto f \circ g,$ the set $\mathcal{H}_b(E,E)$ turns to be a semigroup. From now on $\mathcal{H}_b(E,E)$ will denote the algebraic system $(\mathcal{H}_b(E,E),+,\cdot,\circ)$ (where $+$ and $\cdot$ denote the usual addition and product by scalar on spaces of mappings) endowed with the topology $\tau_b.$ We remark that despite the product defined by the composition is continuous,  $\mathcal{H}_b(E,E)$ is not an algebra in the classical sense (cf. \cite[p. 227]{Ru}), since we may have $h \circ (f + g) \neq h \circ f + h \circ g$ and $\lambda(f \circ g) \neq f \circ \lambda g$ for some $f, g, h \in  \mathcal{H}_b(E,E)$ and $\lambda \in \mathbb{C}.$ It is worth to remark that    the algebra  $\mathcal{L}(E)$ of  bounded  linear operators  on $E$ is embedded in $\mathcal{H}_b(E,E)$.   Borrowing the notion of a (two sided) ideal in an algebra,
 a linear subspace $\mathcal{I}$ of $\mathcal{H}_b(E,E)$ such that
 $f\circ g\circ h \in \mathcal{I}$ whenever $g \in \mathcal{I}$ and $f,h \in \mathcal{H}_b(E,E)$  will be called
 an  \emph{ideal}.
Analogously, an ideal $\mathcal{I}$ in $\mathcal{H}_b(E,E)$ will be called { \it proper} if  $\mathcal{I} \varsubsetneqq \mathcal{H}_b(E,E)$. A proper ideal $\mathcal{I}$ in $\mathcal{H}_b(E,E)$ will be called  maximal if it  is not strictly contained in any other proper ideal.  Of course, given any ideal $\mathcal{I}$ in $\mathcal{H}_b(E,E)$ we have that all constant mappings belong to $\mathcal{I}$ since all constant mappings belong to $\mathcal{H}_b(E,E)$ and $\mathcal{I} \neq \emptyset.$ We recall that an ideal $\mathcal{I}$ in an algebra $\mathcal{A}$ is called non-trivial if $\{0\} \subsetneqq \mathcal{I} \subsetneqq \mathcal{A}.$

 Next we are going to study the structure of the proper closed ideals in $\mathcal{H}_b(E,E).$
  First of all we note that for every complex Banach space $E$,  the space of  constant mappings in $E$  and the subspace $\mathcal{H}_k(E,E)$ of  holomorphic mappings  that map bounded sets into relatively compact sets are proper closed ideals in $\mathcal{H}_{b}(E,E)$ if $E$ is infinite dimensional.

Other simple examples of proper closed ideals in $\mathcal{H}_{b}(E,E)$ follow.
Put
$$
\mathcal{LT}(E):=\{f\in  \mathcal{H}_b(E,E) :\;\text{ for each bounded subset } A\subset E, f(A) \text{ is a limited set in } E\}.
$$
To realize this statement it suffices to recall that any mapping in  $\mathcal{H}_b(E,E)$ maps limited sets into limited sets (\cite[Proposition 15]{Ga}). It is a closed ideal because limited sets have the Grothendieck encapsulating property, i.  e., a set $L$ such that for every $\varepsilon>0$ there is a limited set
$K_\varepsilon\subset E$ such that $L\subset B(0,\varepsilon)+K_\varepsilon,$  is necessarily   limited (see for instance \cite[Lemma 2.4]{GLM}).

The ideal $\mathcal{LT}(E)$ contains $\mathcal{H}_k(E,E).$ However, they may be different: Think of $E=\ell_\infty$ and recall that there is a polynomial $P:\ell_\infty \to c_0\subset \ell_\infty$ that is not weakly compact \cite[Proposition 2]{GG}. On the other hand, $P\in \mathcal{LT}(\ell_\infty)$ since the unit ball of $c_0$ is limited in $\ell_\infty$ (cf. \cite[Theorem 4.28)]{Di}.

Another example of a closed proper ideal in $\mathcal{H}_b(E,E)$: It is known that the Aron-Berner extension
$$
f\in \mathcal{H}_b(E,E)\mapsto \tilde{f}\in \mathcal{H}_b(E^{**},E^{**})
$$
is a well-defined continuous linear mapping. Assume that $E$ is symmetrically regular. Then  $\tilde{f}\circ\tilde{g}=\widetilde{f\circ g}$ for all $f,g\in \mathcal{H}_b(E,E)$ \cite[Corollary 2.2]{ChGarKimMa}. As a consequence, if  $E$ is a non-reflexive symmetrically regular Banach space, then the set
$$
\Xi(E):=\{f\in  \mathcal{H}_b(E,E):\; \tilde{f}(E^{**})\subset E \}
$$
is a closed proper ideal  with $\Xi(E)\cap \mathcal{L}(E)=\mathcal{WK}(E),$ where $\mathcal{WK}(E)$ is the closed ideal of weakly compact operators. Moreover,  $\Xi(E)$ contains the ideal $\mathcal{H}_k(E,E)$. Indeed, according to Davie and Gamelin \cite[Theorem 1]{DG}, for every $z\in E^{**}$ there is a bounded net $(x_i) \subset E$ such that $\tilde{P}(z)=\lim  \bigl( P(x_i) \bigr)  $ for every polynomial $P\in\mathcal{P}(E).$ In addition, for the Aron-Berner extension $\tilde{Q}$ of any $Q\in \mathcal{P}(^nE,E),$ one has $\tilde{Q}(z)=\text{w*-}\lim  \bigl( Q(x_i) \bigr) $. So if $Q$ is  (weakly) compact, we must have $\tilde{Q}(z)\in E.$ To conclude recall that the Taylor series of any $f\in  \mathcal{H}_k(E,E)$ is made up with compact polynomials \cite[Proposition 3.4]{AS}.  In particular the operator $df(x)$ is  compact  for every  $f\in  \mathcal{H}_k(E,E)$ and $ x\in E$, a fact that will be used later in  Proposition \ref{[J]}.

\begin{thm}
	\label{noideal} The following  statements are equivalent:

 1) $E$ does not contain a copy of $\ell_1.$

2) $\mathcal{H}_{wu}(E,E)=\mathcal{H}_k(E,E)$.

3) The  linear space $\mathcal{H}_{wu}(E,E)$ is  an ideal in $\mathcal{H}_{b}(E,E)$.

  \end{thm}
\begin{proof}
 It is always the case that $\mathcal{H}_{wu}(E,E)\subset\mathcal{H}_k(E,E).$
 Every  $f\in \mathcal{H}_k(E,E)$ maps weakly Cauchy sequences $(x_n)\subset E$ into Cauchy sequences. If  this were not the case, there would exist a subsequence $(x_n)$ itself, say, and $\varepsilon >0$ such that $\|f(x_n)-f(x_m)\|>\varepsilon$ for $n\neq m.$ However, since $f$ is compact, $(f(x_n))$ will have a convergent subsequence. Now, under the assumption that $E$  does not contain $\ell_1,$   we obtain that
$f\in\mathcal{H}_{wu}(E,E)$ by  \cite[Proposition 3.3 b)]{AHV}. Thus, $1)$ implies $2).$ Obviously, $2)$ implies $3).$

Suppose that $E$  contains a copy of $\ell_1$.
Then there is a continuous homogeneous polynomial  $P:E \rightarrow \mathbb{C}$ that is not weakly continuous on bounded sets (see \cite[Theorem 4.(e)]{Joaquin} or  \cite[Remark 1.6]{AGGM}).  Choose an element $a \in E \backslash \{0\}$ and define $Q:E\to E$ by $Q(x)=P(x)a.$ It is clear that  $Q$ is an element of $\mathcal{H}_{b}(E,E)$ which is not  weakly uniformly continuous on bounded sets. Pick now $\varphi\in E^*$ such that $\varphi(a)=1$. Since   $Q(x)=\big((\varphi\otimes a)\circ Q\big)(x)$, the  mapping $(\varphi\otimes a)\circ Q$ does not belong to $ \mathcal{H}_{wu}(E,E)$ despite $\varphi\otimes a \in \mathcal{H}_{wu}(E,E).$ This shows that $\mathcal{H}_{wu}(E,E)$ is not an ideal in $\mathcal{H}_{b}(E,E).$ So $3)$ implies $1).$
\end{proof}
It is well-known  that every non-zero ideal  $\mathcal{J}\subset \mathcal{L}(E)$ contains the ideal $\mathcal{F}(E)$ of all     finite rank operators on $E$ (see for instance \cite[p. 160]{A}). Hence every non-zero closed ideal
 in $\mathcal{L}(E)$ contains the closed ideal
 $\mathcal{K}(E)$ of all compact operators on $E$ whenever   $E$ has the approximation property.

 \medskip

The most natural way to connect  ideals  in $\mathcal{L}(E)$ and  ideals in $\mathcal{H}_b(E,E)$ is through differentiation. We study this matter in the sequel.
\begin{prop}
		\label{finite-rank}
	 Let $\mathcal{I}$ be an ideal  in  $\mathcal{H}_b(E,E).$

 1) Then  $\mathcal{I_L}:=\{df(0):f\in\mathcal{I}\}$ is an ideal in $\mathcal{L}(E)$ and $\mathcal{I_L}=\{df(x):f\in\mathcal{I}\;\text{ and } x\in E\}.$

2) If $\mathcal{I}$ is closed and contains some non-constant mapping, then  the operator ideal $\mathcal{F}(E)$ of finite rank operators is contained in $\mathcal{I}$. If moreover,  $E$ has the approximation property, then   $\mathcal{H}_{k}(E,E)\subset \mathcal{I}.$
\end{prop}
\begin{proof}

	1) Obviously, $\mathcal{I_L}$ is a linear subspace of $\mathcal{L}(E).$ If $T,S\in \mathcal{L}(E)$  and $df(0)\in \mathcal{I_L},$ then $d(T\circ f \circ S)(0)=T\circ df(0)\circ S,$ so $\mathcal{I_L}$ is an ideal in $\mathcal{L}(E)$ because $T\circ f \circ S\in \mathcal{I}.$

Notice as well that $\mathcal{I_L}=\{df(x):f\in\mathcal{I}\;\text{ and } x\in E\}$ since for the translation map $\tau_x$ on $E,$ $f\circ \tau_x \in \mathcal{I}$ and $d(f\circ \tau_x)(0)=df(x)\circ Id=df(x).$

2) Since there is a non-constant mapping $f\in  \mathcal{I},$ there must be $x\in E$ such that $df(x)$ is a non null operator.   Now we check that $df(x)\in  \mathcal{I}.$ For every $t\in\mathbb{C},$ the mapping $y\in E \mapsto x+ty\in E$ belongs to $\mathcal{H}_b(E,E)$  and hence, for every $t\in\mathbb{C}\backslash \{0\}$ the mapping  $f_t $  defined  by $f_t(y)= \frac{1}{t}\big(f(x+ty)-f(x)\big)$  belongs to the ideal $\mathcal{I}.$ Further, $\lim_{t\to 0} f_t=df(x)$ uniformly on bounded subsets of $E.$ Since $\mathcal{I}$ is closed, $df(x)\in \mathcal{I}.$
 Hence     $\mathcal{F}(E) \subset   \mathcal{I}$ as
  $ \mathcal{I} \cap \mathcal{L}(E)  $ is a non-zero  ideal  in   $ \mathcal{L}(E).$


 Assume now that $E$ has the approximation property. For each element  $f  \in  \mathcal{H}_{k}(E,E)$, we are going to check that
 $ f \in \overline{ \mathcal{I}}.$ Fix $\varepsilon > 0$ and  a bounded set  $B \subset E$. Since $f(B)$ is relatively compact and $E$ has the approximation property, there is  $T \in \mathcal{F}(E) $ such that $\Vert T(u) - u \Vert < \varepsilon$ for each $u \in f(B)$. So
	$$
	\Vert T(f(x)) - f(x) \Vert < \varepsilon, \sep \forall x \in B.
	$$
	Since $T \in \mathcal{F}(E) \subset \mathcal{I} $ and $f \in \mathcal{H}_{b}(E,E)$ we have that  $ T \circ f \in  \mathcal{I} $ for every $f \in \mathcal{H}_k(E,E)$. Therefore,  $ f \in \overline{ \mathcal{I}},$ and so,  $ \mathcal{H}_k(E,E)\subset \overline{ \mathcal{I}}= \mathcal{I}.$\end{proof}
 In particular, the only  closed ideal $ \mathcal{I}$ in $\mathcal{H}_b(E,E)$ such that $\mathcal{I} \cap \mathcal{L}(E) = \{0\}$ is the ideal  of  constant mappings.
\medskip

 A classical result  due to  J.W. Calkin   \cite[Theorem 1.3]{Ca}  states  that the ideal $\mathcal{K}(\ell_2)$
 of compact linear operators on $\ell_2$ is the only non-trivial, closed ideal in the algebra $\mathcal{L}(\ell_2)$. Later this result was generalized to $\ell_p $ ($1 \leq p < \infty$) and $c_0$ by I.C. Gohberg,  A.S. Markus  and I.A. Fel'dman    \cite[Theorem 5.1]{FGM} (see also \cite[Theorem 5.2.2]{Pi}).  Next we are going to study the structure of the proper closed  ideals in $\mathcal{H}_b(E,E)$ in
  case that  $E=\ell_p\;(1 \leq p < \infty)$ or  $E= c_0$.  We will see that  the structure of the proper closed ideals in $\mathcal{H}_b(E,E)$ may be  quite different from the structure of the non-trivial closed ideals in the algebra $\mathcal{L}(E).$

 \begin{prop}
 	\label{[J]}
 If $E$ is a complex Banach space and $\mathcal{J}$ is a non-trivial closed ideal in $\mathcal{L}(E),$ then the set
 $$
 [\mathcal{J}]:= \{f \in \mathcal{H}_{b}(E,E) : df(x) \in \mathcal{J} \;\; \text {for all} \;\; x \in E\}
 $$
 is a   proper closed ideal in $\mathcal{H}_{b}(E,E)$ containing $\mathcal{J}.$ For the ideal of compact operators $\mathcal{K}(E)$,  the inclusion $[\mathcal{K}(E)]\supset \mathcal{H}_k(E,E)$ holds.

In particular, if $E=c_0$ or $\ell_p \;\;(1 \leq p < \infty)$, then $[\mathcal{K}(E)]$  contains strictly $\mathcal{H}_k(E,E)$.
\end{prop}
\begin{proof}
The inclusion   $\mathcal{J} \subset [\mathcal{J}]$ is trivially satisfied. Let us show that $[\mathcal{J}]$ is a closed  ideal in  $\mathcal{H}_{b}(E,E)$.  Indeed, for every
$f,g \in [\mathcal{J}]$, $\lambda \in \C$ and
$x \in E$,  we have $d(f \circ g)(x)= df(g(x)) \circ dg(x) \in \mathcal{J}$ and  $d(f + \lambda g)(x)= df(x) + \lambda dg(x) \in \mathcal{J}$ since by hypothesis $\mathcal{J}$ is an ideal in $\mathcal{L}(E).$  Moreover, given any $(f_n)_n \in [\mathcal{J}]$  such that  $\bigl(f_n \bigr)\rightarrow f$ in $\mathcal{H}_{b}(E,E),$  we have that $(df_n(x))_n \subset \mathcal{J}$ and , according to Cauchy's integral formula, $\bigl(df_n(x) \bigl) \rightarrow df(x)$ in $\mathcal{L}(E)$ for all $x \in E.$ This implies $f \in [\mathcal{J}]$ since, by hypothesis, $\mathcal{J}$ is closed in $\mathcal{L}(E).$   Finally, $ [\mathcal{J}] \subsetneqq  \mathcal{H}_{b}(E,E)$ as clearly $Id_E \not\in [\mathcal{J}].$

We noticed above that in case that  $f\in  \mathcal{H}_k(E,E) \text{ and } x\in E,$ then  $df(x)$ is a compact operator. Thus $ \mathcal{H}_k(E,E) \subset [\mathcal{K}(E)].$
       This completes the proof of the first statement.

Next  we prove  that
     $[\mathcal{K}(E)] \varsupsetneqq \mathcal{H}_k(E,E),$  whenever $E=c_0$ or $E=\ell_p \;\;(1 \leq p < \infty).$

If $E=c_0,$ let $P: c_0 \rightarrow c_0$ be defined by $ P((x_n)_n)=(x_n^2)_n$ for each $(x_n)_n \in c_0.$ Clearly $P \in \mathcal{H}_{b}(c_0,c_0).$ If $(e_n)$  is the canonical basis of $c_0$, then $P(e_n)=e_n$ for each $n\in \N$  and so $P \not\in \mathcal{H}_k(c_0,c_0).$ Moreover, for each $x=(x_n)_n \in c_0$ the mapping $dP(x): c_0 \rightarrow c_0$ is given by $dP(x)(y)=(2x_ny_n)$ for every $y=(y_n)_n \in c_0.$ We claim that $dP(x) \in \mathcal{K}(c_0).$ Indeed, given $\varepsilon  > 0$ there exists $n_0 \in \mathbb{N}$ such that $\vert x_n \vert < \varepsilon $ for every $n > n_0.$ So,  for every $y=(y_n)_n \in B_{c_0},$ we have $dP(x)(y)=(2x_ny_n)_n = (2x_1y_1, \ldots, 2x_{n_0}y_{n_0}, 0 , 0,\ldots) + (0, \ldots ,0,2x_{n_0+1}y_{n_0+1}, 2x_{n_0+2}y_{n_0+2}, \ldots) \in K + 2\varepsilon  B_{c_0}$ where $K=\{(2x_1y_1, \ \ldots,2x_{n_0}y_{n_0}, 0 , 0,\ldots): y=(y_n)_n \in B_{c_0} \}$ is compact. So, $dP(x)(B_{c_0})$ is relatively compact i.e., $dP(x) \in \mathcal{K}(c_0)$ and this is true for all $x \ c_0.$ Consequently $P \in [\mathcal{K}(c_0)] \backslash \mathcal{H}_k(c_0,c_0)$ and so $[\mathcal{K}(c_0)]$  is a proper closed ideal  in $\mathcal{H}_{b}(c_0,c_0)$ that contains strictly the ideal $\mathcal{H}_k(c_0,c_0).$

If $E=\ell_1,$ the same argument used above shows that $[\mathcal{K}(\ell_1)]$  is a proper closed ideal  in $\mathcal{H}_{b}(\ell_1,\ell_1)$ that contains strictly the ideal $\mathcal{H}_k(\ell_1,\ell_1).$

If $E=\ell_p \;\;(1 < p < \infty ),$
take $q \in \mathbb{N}$ such that $q \geq p$ and define $P: \ell_p \rightarrow \ell_p$ by $ P((x_n)_n)=(x_n^q)_n$ for every $(x_n)_n \in \ell_p$. It is clear that $P \in \mathcal{H}_{b}(\ell_p,\ell_p)$. The polynomial $P$ preserves the canonical basis of $\ell_p$  and so $P \not\in \mathcal{H}_k(\ell_p, \ell_p)$. Moreover, for each $x=(x_n)_n \in \ell_p$ the mapping $dP(x): \ell_p \rightarrow \ell_p$ is given by $dP(x)(y)=(qx_n^{q-2} x_n y_n)_n$ for every $y=(y_n)_n \in \ell_p.$ We claim that $dP(x) \in \mathcal{K}(\ell_p).$
Indeed, given $0 < \varepsilon  < 1$ there exists $n_0 \in \mathbb{N}$ such that $
 \bigl( \sum _{k=n_0+1} ^\infty  \vert x_k \vert ^q \bigr) ^\frac{q-1}{q} < \frac{\varepsilon}{q}$. For every $y=(y_n)_n \in B_{\ell_p},$ we have $dP(x)(y)=(q x_n^{q-2} x_n y_n)_n = q(x_1^{q-1} y_1, \ldots, x_{n_0}^{q-1}y_{n_0}, 0 , 0,\ldots) + q(0, \ldots,0, x_{n_0+1}^{q-2}x_{n_0+1}y_{n_0+1}, x_{n_0+2}^{q-2}x_{n_0+2}y_{n_0+2}, \ldots).$
Note that the subset $K$ given by
 $$
 K=\{(qx_1^{q-1} y_1, \ldots, qx_{n_0}^{q-1}y_{n_0}, 0 , 0,\ldots): y=(y_n)_n \in B_{\ell_p}\}
 $$
is compact and
$$
\sum_{k=n_0+1}^\infty \vert x_k ^{ q-1}  y_k \vert \le \Bigl(\sum_{k=n_0+1}^\infty \vert x_k \vert^q \Bigr)^{\frac{q-1}{q}} \Bigl(\sum_{k=n_0+1}^\infty \vert y_k \vert^q \Bigr)^{\frac{1}{q}} \le
\frac{\varepsilon}{q} \Vert y \Vert _q \le
\frac{\varepsilon}{q}.
$$
Hence, for every $y=(y_n)_n \in B_{\ell_p},$ we have $dP(x)(y) \in K +   \varepsilon  B_{\ell_p}$ and, consequently,  $dP(x)(B_{\ell_p})$ is relatively compact. We proved that  $dP(x) \in \mathcal{K}(\ell_p)$ every  $x \in \ell_p$, that is, $P \in [\mathcal{K}(\ell_p)]$. Since $P \notin  \mathcal{H}_k(\ell_p,\ell_p)$,  $[\mathcal{K}(\ell_p)]$  is a proper closed ideal in $\mathcal{H}_{b}(\ell_p,\ell_p)$ that contains strictly the ideal $\mathcal{H}_k(\ell_p,\ell_p)$.
\end{proof}

\begin{prop}
	\label{[M]}
	 If $\mathcal{I}$ is an ideal in $\mathcal{H}_{b}(E,E),$  then  $\mathcal{I}\subset [\mathcal{I_L}].$ If $\mathcal{J}$ is an ideal in $\mathcal{L}(E),$ then $[\mathcal{J}]_{\mathcal{L}}=\mathcal{J}$. Further, if there is a largest non-trivial ideal  $\mathcal{M}$ in $\mathcal{L}(E),$ then $[\mathcal{M}]$ is closed and it is the largest proper ideal in $\mathcal{H}_{b}(E,E)$.
\end{prop}
\begin{proof}
If $f\in\mathcal{I},$ then according to Proposition \ref{finite-rank},  $df(x)\in\mathcal{I_L}$ for all $x\in E,$ that is $f\in [\mathcal{I_L}].$
The second assertion is also easily checked.
For the third statement, notice that for any ideal $\mathcal{I},$ the ideal $\mathcal{I_L}$ lies inside $\mathcal{M},$ hence $\mathcal{I}\subset[\mathcal{I_L}]\subset [\mathcal{M}]$.   That $ [\mathcal{M}]$ is closed follows from the fact that $\mathcal{M}$ is closed as the largest non-trivial ideal   in $\mathcal{L}(E)$ plus Proposition \ref{[J]}.
\end{proof}

Notice that  in view of the  previous result we have that
$H_k(E,E)_\mathcal{L}=\mathcal{K}(E)=[\mathcal{K}(E)]_\mathcal{L}.$  Since it might happen that   $H_k(E,E) \ne [\mathcal{K}(E)]$, it turns out that  an ideal $\mathcal{I}
\subset \mathcal{H}_{b}(E,E)$ cannot in general be recovered  from the derivatives part $\mathcal{I}_\mathcal{L}.$
\medskip

Proposition \ref{[M]} allows us to obtain examples of spaces $E$  where $[\mathcal{K}(E)]$ is closed and it is the largest  proper ideal in $\mathcal{H}_{b}(E,E).$ Indeed,  as proved in \cite[Theorem 5.1]{FGM} (see also \cite[Theorem 5.2.2]{Pi}) in the cases $E=c_0$ or $\ell_p \;\;(1 \leq p < \infty)$ we have that  $\mathcal{K}(E)$ is the largest closed non-trivial ideal in $\mathcal{L}(E)$ and so $[\mathcal{K}(E)]$ is the largest
proper closed ideal in $\mathcal{H}_{b}(E,E)$.
 Aside the closed ideal of constant mappings, there are at least two proper closed ideals for these spaces, namely  $\mathcal{H}_{k}(E,E)$ and $[\mathcal{K}(E)],$ and any closed ideal $\mathcal{I}$ lies between them.  We do not know  whether there exists  any  other proper closed ideal in these cases.

 Argyros and Haydon constructed in \cite{AH} an hereditary indecomposable $\mathcal{L}_\infty$-space $\mathfrak{X}_K$ that solves the scalar-plus-compact problem and consequently this space has the property that $\mathcal{K}(\mathfrak{X}_K)$ is the only non-trivial closed ideal in $\mathcal{L}(\mathfrak{X}_K)$ (see \cite[Theorem 7.4]{AH}). So $[\mathcal{K}(\mathfrak{X}_K)]$ is closed and it is  the largest proper ideal in $\mathcal{H}_{b}(\mathfrak{X}_K,\mathfrak{X}_K).$

 Also the spaces $E=\ell_\infty$ and $E=J_p \; (1<p<\infty)$ (where $J_p$ is the p-th James space), enjoy the property that $\mathcal{WK}(E)$ is the largest non-trivial closed ideal in $\mathcal{L}(E)$ (see \cite[p. 253]{LL}  and  \cite[Theorem 4.16]{La}, respectively) and so, by  Proposition \ref{[M]},  $[\mathcal{WK}(\ell_\infty)]$ is closed and it is  the largest proper ideal in $\mathcal{H}_{b}(\ell_\infty,\ell_\infty)$ and $[\mathcal{WK}(J_p)]$ is closed and it is  the largest proper ideal in $\mathcal{H}_{b}(J_p,J_p).$

 It is said that $S \in \mathcal{L}(E)$ factors through $T \in \mathcal{L}(E)$ if there exist $A,B \in \mathcal{L}(E)$ such that $S=ATB.$ The set of all $T \in \mathcal{L}(E)$ so that the identity operator $Id_E$ on $E$ does not factor through $T$ is usually denoted by $\mathcal{M}_E.$ It is easy to see that $\mathcal{M}_E$ is a proper ideal in $\mathcal{L}(E)$ whenever  $\mathcal{M}_E+\mathcal{M}_E \subset \mathcal{M}_E.$   Moreover, if $\mathcal{M}_E$  is an ideal in $\mathcal{L}(E)$ then it is automatically the largest proper ideal in $\mathcal{L}(E)$ and hence is closed.
 Indeed, given any proper ideal $\mathcal{I}$ in $\mathcal{L}(E),$ clearly $ATB \neq Id_E$ for every $T \in \mathcal{I}$ and for all $A,B \in \mathcal{L}(E).$ Hence, $\mathcal{I} \subset \mathcal{M}_E.$

It follows from Proposition \ref{[M]} that $[\mathcal{M}_E]$ is the largest closed proper ideal in $\mathcal{H}_{b}(E,E)$ whenever
 $\mathcal{M}_E$ is an ideal in $\mathcal{L}(E).$ Moreover, clearly $\mathcal{[M}_E] \cap\mathcal{L}(E)=\mathcal{M}_E$.
Besides the Banach spaces mentioned above, there are others for which $[\mathcal{M}_E]$ is the largest closed proper ideal in $\mathcal{H}_{b}(E,E).$ For instance,
if $E=(\bigoplus_{k=1}^\infty \ell_2^k)_{c_0},$  the closure of the ideal of operators that factor
through $c_0\;\;$ coincides with $\mathcal{M}_E$
\cite[Theorem 5.5]{LLR}.
Also if $E= (\bigoplus_{k=1}^\infty \ell_2^k)_{\ell_1},$
  the closure of the ideal of operators that factor through $\ell_1\;$ coincides with $\mathcal{M}_E$ (\cite[Corollary 2.11]{LSZ}).
There are quite a number of spaces $E$ for which $\mathcal{M}_E$ is the unique maximal ideal in
$\mathcal{L}(E)$ and hence the largest one. For instance,
 this is the case of $L_p \;\;(1 \leq p < \infty),\;\mathcal{C}([0,\omega  _1]),\;
 (\Sigma \ell_q)_{\ell_p} \;\;(1 \leq q < p < \infty)\;$ and the Lorentz sequence spaces $\;d_{\omega,p}\;\; (1 \le p \; <\infty) $
(cf. \cite{DJS}, \cite{KL}, \cite{CJZ} and \cite{KPSTT}).

We say that the $Id_E$ factors through $f \in \mathcal{H}_b(E,E)$ if there exist $g,h \in \mathcal{H}_b(E,E)$ such that $Id_E = g \circ f \circ h.$
As in the linear case, we can show that the set $\mathcal{I}_E$ of all $f \in \mathcal{H}_b(E,E)$ so that the $Id_E$ does not factor through $f$ is the largest proper ideal in $\mathcal{H}_b(E,E)$ whenever it is closed under addition in $\mathcal{H}_b(E,E)$.  By using the chain rule, we get that $\mathcal{I}_E \cap \mathcal{L}(E) = \mathcal{M}_E$ and that $\mathcal{I}_E\supset [\mathcal{M}_E].$ Further, if  $\mathcal{I}_E$ is an ideal, then $\mathcal{M}_E$ is an ideal as well and consequently  $\mathcal{I}_E=[\mathcal{M}_E]$.

\bigskip

It is immediate from Proposition \ref{cont-ac} that for any ideal $\mathcal{I}\subset \mathcal{H}_b(E,E),$ its closure $\overline{\mathcal{I}}$ is also an ideal.  Next we will show that the same property  holds for $\overline{\mathcal{I}}^{\tau_0} \cap \mathcal{H}_{b}(E,E).$

 \begin{lem}
\label{le}
Let  $E$ be a Banach space and $U$  an open subset of $E$. If $g \in \mathcal{H}(U,E)$ and $(g_ i)_{i \in I}$
is a net in $\mathcal{H}(U,E)$  such that $(g_i) \stackrel{\tau_0}{\rightarrow} g$, then given any compact
subset $K$ of $U$ and any open set $V$ containing $g(K)$,  there exists $i_0 \in I$ such that $g_i(K) \subset
V $,   for all $i \ge  i_0$.
\end{lem}

\begin{proof}
Assume that $K \subset U$ is a compact set.  Since $g(K)$ is compact, and $V$ an open set such that $g(K)
\subset V$, there is $\varepsilon > 0$ such that $g(K) + \varepsilon B_E \subset V$. By assumption  $(g_i)
\stackrel{\tau_0}{\rightarrow} g$, so there is $i_0 \in I$ satisfying that
$$
\Vert g_i(x)- g(x) \Vert < \varepsilon, \;\; \text{for all};\;  i \ge i_0,\;\; \text{for all}\;\; x \in K.
$$
Hence for every $i \ge i_0$ and $x \in K$ we have that
$$
g_i(x) \in g(x)+ \varepsilon  B_E \subset g(K)  + \varepsilon B_E \subset V,
$$
that is $g_i(K) \subset V$ for every $i \ge i_0$.
\end{proof}

\begin{prop}\label{comp-ab}
 If $E$ is a complex Banach space and $\mathcal{I}$ is an ideal in $\mathcal{H}_{b}(E,E),$ then $\mathcal{W}= \overline{\mathcal{I}}^{\tau_0} \cap \mathcal{H}_{b}(E,E)$ is also an ideal in $\mathcal{H}_{b}(E,E).$ If
  moreover,  $\mathcal{I}=[\mathcal{J}]$ for a non-trivial $\tau_0$-closed ideal
  $\mathcal{J}\subset\mathcal{L}(E)$, then $\mathcal{W}$ is also a proper ideal.
  \end{prop}
 \begin{proof}
It is easy to see that  $f + g \in \mathcal{W}$ whenever $f \in \mathcal{W}$ and $g \in \mathcal{W}$.  Now, take any $f \in  \mathcal{W}$ and $g \in \mathcal{H}_{b}(E,E)$ and choose any net $(f_i)_{i \in I}$ in $\mathcal{I}$ such that $ \bigl( f_i \bigr) \stackrel{\tau_0}{\rightarrow} f.$  Clearly $(f_i \circ g)_{i \in I} \subset \mathcal{I} $ and  $ \bigl( f_i \circ g  \bigr) \stackrel{\tau_0}{\rightarrow} f \circ g$ since $g(K)$ is a compact subset of $E$ whenever $K$ is a compact subset of $E.$  So, $f \circ g \in \mathcal{W}.$ We just proved that $f \circ g \in \mathcal{W}$ whenever $f \in  \mathcal{W}$ and $g \in \mathcal{H}_{b}(E,E).$ In particular we proved that $f \circ g \in \mathcal{W}$ whenever $f \in \mathcal{W}$ and $g \in \mathcal{W}$. Next we are going to prove that  $g \circ f \in \mathcal{W}$. Indeed, given a compact subset $K$ of $E$,  there is an open and bounded set $V \subset E$ such that $f(K) \subset V$ and, by Lemma  \ref{le},  there exists $i_0 \in I$ such that $f_i(K) \subset V$ for all  $i \geq i_0.$ Moreover, $g$ is uniformly continuous  on $V$ and so, given any $\varepsilon  > 0$ there exists a $\delta > 0$ such that $\Vert g(u) - g(v) \Vert < \varepsilon $ whenever $u,v \in V$ and $\Vert u - v \Vert < \delta.$ The fact that $ \bigl( f_i  \bigr) \stackrel{\tau_0}{\rightarrow} f$ guarantees the existence of  $i_1 \in I, \;\; i_1 > i_0,$ such that  $\Vert f_i(x) - f(x) \Vert < \delta$ for every $x \in K.$ Consequently, as $f_i(K) \subset V$ for every   $i \geq i_0$ and $f(K) \subset V,$ we have
  $$
  \Vert g(f_i(x)) - g(f(x)) \Vert < \varepsilon  \;\; \text{for all} \;\; i > i_1.
  $$
Hence, $ \bigl( g \circ f_i \bigr) \stackrel{\tau_0}{\rightarrow} g \circ f$ and this implies $g \circ f \in \mathcal{W}.$

In the case that $\mathcal{I}= [\mathcal{J}]$, for a non-trivial $\tau_0$-closed ideal $ \mathcal{J} \subset \mathcal{L}(E)$, if $Id_E\in \mathcal{W},$ there would exist a net $(f_i)_{i \in I}\subset [\mathcal{J}]$ converging to $Id_E$ on compact subsets of $E.$ Then, by Cauchy's integral formula, also $(df_i(0))_{i \in I}$ would $\tau_0$-converge to $Id_E.$
So we would be led to the contradiction $Id_E\in \overline{\mathcal{J}}^{\tau_0}=\mathcal{J}$.
  \end{proof}

\begin{remark}
	If $E$  is a complex Banach space with the approximation property, then no non-trivial ideal $\mathcal{J}$ in $\mathcal{L}(E)$ can be $\tau_0$-closed.  Indeed, as we already recalled $\mathcal{F}(E) \subset \mathcal{J}.$ Hence, if $E$ has the approximation property, then $Id_E \in \overline{\mathcal{J}}^{\tau_0}$ and thus $\overline{\mathcal{J}}^{\tau_0}=\mathcal {L}(E)$.
	Analogously no proper ideal $\mathcal{I}$ in $\mathcal{H}_{b}(E,E)$  containing a non constant mapping can be $\tau_0$-closed
	if $E$ has the approximation property. Otherwise, assume that $\mathcal{I}$ is  such a  $\tau_0$-closed ideal in $\mathcal{H}_{b}(E,E).$ Then  $\mathcal{I}\cap \mathcal{L}(E)$   is a non-trivial ideal in  $ \mathcal{L}(E)$. Hence
	$Id_E\in\overline{\mathcal{I}\cap \mathcal{L}(E)}^{\tau_0}\subset \overline{\mathcal{I}}^{\tau_0}  \cap \mathcal{H}_{b}(E,E) = \mathcal{I}$ and this leads  to the contradiction $ \mathcal{I} =   \mathcal{H}_{b}(E,E).$  \end{remark}

\section{Negative results }

Let $H^\infty(B_E,E)$ denote the linear space of elements of $\mathcal{H}(B_E,E)$ which are bounded on $B_E,$ endowed with the sup-norm. It is well known that $H^\infty(B_E,E)$ is a Banach space. If $E=\mathbb{C},$ we will write $H^\infty$ instead of  $H^\infty(B_E,E)$.

\begin{remark}
The composition operation is not separately continuous on the unit ball of $H^\infty.$ Consequently, neither is on  the unit ball of $H^\infty(B_E,E).$
\end{remark}
\begin{proof} Let $(z_n)\subset (0,1)$ be an  increasing interpolating sequence for $H^\infty$.  Pick $f\in H^\infty$ such that $f(z_{2m})=1$ and $f(z_{2m-1})=0$,  for every natural number $m$.

Let $a_m=\frac{z_{m-1}}{z_m},$ then $a_m<1$ and $\lim_m (a_m)=1.$ Thus, the sequence $\bigl( g_m(z) \bigr):=\bigl( a_mz \bigr)$  converges in $H^\infty $ to $g(z)=z$. However, the sequence  $ \Bigl( \frac{f}{\|f\|}\circ g_m\Bigr) $ does not converge uniformly to $\frac{f}{\|f\|},$ because
$$
\Big|\Big(\frac{f}{\|f\|}\circ g_m\Big)(z_m)-\frac{f}{\|f\|}(z_m)\Big|=\Big|\frac{f(a_mz_m)-f(z_m)}{\|f\|}\Big|=
$$
$$
\Big|\frac{f(z_{m-1})-f(z_m)}{\|f\|}\Big|=\frac{1}{\|f\|}.$$

For the general case, pick two unitary vectors $e\in E$ and $e^*\in E^*$ such that $e^*(e)=1,$ and consider the mappings $F(x)=\frac{f}{\|f\|}(e^*(x))e$ and $G(x)=g(e^*(x))e.$  Since $(F\circ G)(x)=\frac{f}{\|f\|}(g(e^*(x))),$ and the function $x\in B_E \mapsto e^*(x)\in\triangle$ is onto, the same argument as above yields the statement.
\end{proof}

A paper due to K. Is\'eki (cf. \cite{Iz2}) presents several results on the structure of compact topo\-lo\-gi\-cal semigroups.  One of these results states that any  nonempty compact semigroup contains at least one idempotent. This fact was used, for instance, by Shields  and by Suffridge  to prove the existence of a common fixed point for all the elements of a commuting family $\mathcal{A}$ of continuous (cf. \cite[Theorem 1]{S}) and of analytic (cf. \cite[Theorem 4]{Su}) mappings under certain conditions.  When attempting to extend this kind of results to convenient semigroups of holomorphic self-mappings with the composition operation defined  on infinite dimensional Banach spaces one faces with the failure of Montel theorem in case of the $\tau_0$-topology.
So, we looked for a more convenient topology. In case $E$ is the  dual  of a complex Banach space
$F,$ Kim and Krantz \cite{KK} defined a very natural topology called
the compact-weak*-open topology  (and denoted by $\tau_0^\ast$)
as follows.
 For  an open set $U\subset E,$  the topology $\tau_0^\ast$ in $\mathcal{H}(U,F^*)$ is the topology generated by the family $\{p_{K,L} : K \subset U
	\;\text{compact},\; L \subset F \; \text{finite} \}$ of seminorms defined in $\mathcal{H}(U,F^*)$ by
	$$
	p_{K,L}(f)= \displaystyle\max_{x \in K, y \in L} \vert f(x)(y) \vert .
	$$


Unfortunately, despite the topology $\tau_0^\ast$ seemed to be a good
choice to deal with our problem, it turns out that  the composition mapping is not continuous in the space of holomorphic mappings we were interested, as we show next.

\begin{prop}
\label{conv-tau-0-*-acotada}
Let $F$ be a Banach space and  write $E=F^*$.  Assume that  $(f_\alpha)_{\alpha \in \Lambda}$ is a net
in $\mathcal{H} (B_E, B_E)$, such that $f_\alpha (x) (y) \to f
(x)(y)$ for each $x \in B_E$ and $ y \in F$  for some $f\in \mathcal{H} (B_E, B_E).$ Then $ \bigl( f_\alpha \bigr)  \stackrel{\tau_0^*}{\to} f.$
\end{prop}
\begin{proof}
As usual, for each $y \in F$ let $\tilde{y}$ denote the element of $F^{\ast \ast }$ defined by $\tilde{y}(\varphi)=\varphi(y)$ for every $\varphi \in F^\ast $.  Clearly $ \bigl( f_\alpha \bigr)  \stackrel{\tau_0^*}{\to} f$ if and only if $ \bigl( \tilde{y} \circ f_\alpha \bigr)  \stackrel{\tau_0}{\to} \tilde{y} \circ f$ for every $y \in F.$ By hypothesis, $(\tilde{y} \circ f_\alpha )_{\alpha \in \Lambda}$ converges pointwise to $\tilde{y} \circ f$ in $B_E$ for each $y \in F$. Moreover, for each $y \in F$ it is clear that $\{\tilde{y} \circ f_\alpha: \alpha \in \Lambda\}$ is an equicontinuous subset of $\mathcal{C}(B_E)$ and so, by \cite[Proposition 9.11]{Mu}, $\tau_0$ induces in $\{\tilde{y} \circ f_\alpha: \alpha \in \Lambda\}$ the topology of pointwise convergence. This implies that $\bigl( f_\alpha \bigr)  \stackrel{\tau_{0}^*}{\longrightarrow } f.$
\end{proof}

\begin{thm}
\label{cont-tau-0-*}
Assume that $F$ is an infinite-dimensional  Banach space and write $E=F^*.$ Then the mapping
$$
(f,g) \in \big(\mathcal{H}(B_{E},B_{E}), \tau_0^\ast\big) \times \big(\mathcal{H}(B_{E} ,B_{E} ),\tau_0^\ast\big)
\mapsto f \circ g \in \big(\mathcal{H }(B_{E} ,B_{E} \big), \tau_0^\ast)
$$
is not  continuous.
\end{thm}

\begin{proof}
Since $F$ is infinite-dimensional,  the set  $S_F$ is dense in the closed unit ball of $F$ for  the weak
topology. Hence there is a net $(z_\alpha )$ in $F$ such that
\begin{equation}
\label{z-alpha}
\Vert z_\alpha  \Vert = \frac{1}{4} \sep \text{for every}\  \alpha  \seg \text{and} \seg  (z_\alpha )
\stackrel{w}{\to} 0.
\end{equation}
So there is a net $(z_{\alpha} ^{*})$ in $E$ satisfying that
\begin{equation}
\label{z-alpha-*}
\Vert z_{\alpha} ^{*} \Vert = \frac{1}{4} \seg \text{and} \seg  z_{\alpha }^* (z_\alpha )= \frac{1}{16} \seg
\text{for every}\  \alpha .
\end{equation}
Since the  closed unit ball of $E$ is weak$^*$-compact, and passing to a subnet, if needed, we can also
assume that the net $(z_{\alpha} ^{*})$ converges in the weak$^*$ topology of $E$ to some element $z^* \in
E$ satisfying that $\Vert z^* \Vert \le \frac{1}{4}$.

Now choose  elements $y_{0}^* \in \frac{1}{4} S_{E}$  and $y_0 \in \frac{1}{4} S_F$ satisfying also that
\begin{equation}
\label{y0-y0*}
y_{0}^* \ne - z^{*}, \seg   (z^* + y_{0}^*)(y_0) \ne 0 .
\end{equation}

Define the nets $(y_\alpha )$ in $F$  and $(y_{\alpha }^*)$  in $E$ by
$$
y_\alpha = z_\alpha  +y_{0}, \seg  y_{\alpha }^*= z_{\alpha }^* +y_{0}^*, \sep \forall \alpha .
$$
In view of \eqref{z-alpha}   and \eqref{z-alpha-*} it is clear that for each $\alpha $ we have
\begin{equation}
\label{norma-yn-yn*}
\Vert y_\alpha  \Vert \le \Vert z_\alpha  \Vert + \Vert y_0 \Vert \le  \frac{1}{2}
 \seg  \text{and}
\seg \Vert y_{\alpha}^*  \Vert \le \Vert z_ { \alpha }^* \Vert + \Vert y_{0}^* \Vert \le \frac{1}{2}.
\end{equation}

It is also satisfied that
\begin{equation}
\label{conv-y-alpha-y*}
(y_\alpha ) \stackrel{w}{\to}  y_0, \seg (y_{\alpha}^*) \stackrel{w^*}{\to} z^* +  y_{0}^*, \seg
\bigl(y_{\alpha}^* (y_\alpha )\bigr) \to \frac{1}{16} + (z^{*} + y_{0}^* )(y_0)  .
\end{equation}

Let us consider the net of mappings $(f_\alpha )$ given by
$$
f_\alpha  (y^*)  = y^* ( y_\alpha ) y_\alpha  ^* \seg (y^* \in B_{E}, \alpha \in \Lambda  ).
$$
By \eqref{norma-yn-yn*} we have that $f_\alpha  \in \mathcal{H}(B_{E}, B_{E})$. For every $ y^* \in
B_{E}$ and $y \in  F$, in view of \eqref{conv-y-alpha-y*} we also have that
$$
f_\alpha  (y^*)  (y) =  y^* ( y_\alpha ) y_\alpha  ^*(y) \to y^* (y_0)   \bigl( z^{*} + y_{0}^* \bigr)  (y).
$$
 By using Proposition
\ref{conv-tau-0-*-acotada}, we obtain that $\bigl( f_\alpha \bigr)   \stackrel{\tau_{0}^*}{\to } f$, where $f$ is the
mapping given by
$$
f (y^*)  =  y^* (y_0)   \bigl( z^{*} + y_{0}^* \bigr)   \seg \text{for all } y^* \in B_{E},
$$
that clearly belongs to $\mathcal{H} (B_{E}, B_{E})$.

Now we will check that  $(f_\alpha  \circ f_\alpha )$  does not converge to  $f \circ f$  in the topology
$\tau_{0}^*$. Let us fix $y^* \in B_{E}$ and $y \in F$. On one hand,  by \eqref{conv-y-alpha-y*} we have that
$$
\bigl( (f_\alpha  \circ f_\alpha  )(y^*) \bigr) (y)= f_\alpha  (y^* ( y_\alpha ) y_\alpha  ^*) (y)=   y^* ( y_\alpha )
y_\alpha ^* (y_\alpha ) y_{\alpha}^*(y) \to  y^* (y_0) \Bigl( \frac{1}{16}  + \bigl( z^{*} + y_{0}^* \bigr)
(y_0) \Bigr) \bigl( z^* + y_{0}^* \bigr) (y).
$$
On the other hand,
$$
\bigl( (f \circ f )(y^*) \bigr) (y)= f \bigl(  y^* ( y_0) \bigl( z^{*} + y_{0}^* \bigr) \bigr)(y)=   y^* ( y_0)
 \bigl( z^{*} + y_{0}^* \bigr) (y_0)  \bigl( z^{*} + y_{0}^* \bigr) (y).
$$
Since  $y_0 \ne 0$ and $z^* + y_{0}^* \ne 0$, we can choose elements $y^* \in B_{E}$ and  $y \in F$  such
that $y^* (y_0) \ne 0$ and  $(z^* + y_{0}^*)(y)  \ne 0$. Hence  $(f_\alpha \circ f_\alpha )$ does not
converge to $f \circ f$ in the topology $\tau_{0}^*$.
\end{proof}

\bibliographystyle{amsalpha}

\begin{thebibliography}{20}
\bibitem{A}
 G.R. Allan, \textit{Introduction to Banach Spaces and Algebras,} Oxford Graduate Texts in Mathematics {\bf 20}, Oxford University Press 2011.
	
\bibitem{AH}
S.A. Argyros and R.G. Haydon, \textit{A hereditary indecomposable $\mathcal {L}_\infty$-space that solves the scalar-plus-compact problem}, Acta Math.  \textbf{206} (1) (2011), 1--54.
\bibitem{AGGM}
R.M. Aron, P. Galindo, D. Garc\'{i}a and M. Maestre,  \textit{Regularity and algebras of analytic functions in infinite dimensions,}
Trans. Amer. Math. Soc. \textbf{348} (2) (1996),
543--559.
\bibitem{AHV}
R.M. Aron, C. Herv\'es and M. Valdivia, \textit{Weakly continuous
	mappings on Banach spaces}, J. Funct. Anal.  \textbf{52}  (2) (1983),
189--204.
\bibitem{AS}
R.M. Aron and M. Schottenloher, \textit{Compact holomorphic mappings on Banach spaces and the approximation property,} J. Funct. Anal.  \textbf{21} (1) (1976), 7--30.
\bibitem{Ca}
J.W. Calkin, \textit{Two-sided ideals and congruences in the ring of bounded operators in Hilbert space}, Ann. of Math. \textbf{42} (2) (1941), 839--873.
\bibitem{CJZ}
D. Chen, W.B. Johnson and B. Zeng, \textit{Commutators on $(\sum \ell_q)_p$}, Studia Math. \textbf{206} (2) (2011), 175--190.
\bibitem{ChGarKimMa}
Y.S. Choi, D. Garc\'{i}a, S.G. Kim and M. Maestre,  \textit{Composition, numerical range and Aron-Berner extension},
 Math. Scand. \textbf{103} (1) (2008), 97--110.
\bibitem{DG}
A.M. Davie and T.W. Gamelin, \textit{A theorem on polynomial-star approximation}, Proc. Amer. Math. Soc. {\bf 106}  (2) (1989), 351--356.
\bibitem{Di}
S. Dineen, \textit{Complex Analysis in Locally  Convex Spaces}, North-Holland Math. Studies {\bf 57}, North-Holland 1981.
\bibitem{DJ}
D. Dosev and W.B. Johnson, \textit{Commutators on $\ell_\infty$}, Bull. London Math. Soc. \textbf{42} (1) (2010), 155--169.
\bibitem{DJS}
D. Dosev, W.B. Johnson and G. Schechtmann, \textit{Commutators on $L_p, \; 1 \leq p < \infty$}, J. Amer. Math. Soc. \textbf{26} (1) (2013), 101--127.
\bibitem{Ga}
 P. Galindo, \textit{ Polynomials and limited sets,}  Proc. Amer. Math. Soc.
 {\bf 124} (5)  (1996), 1481--1488.
\bibitem{GLM}
P. Galindo, M.L. Louren\c{c}o and L.A. Moraes, \textit{Compact and
	weakly compact homomorphisms on Fr{\'e}chet algebras of holomorphic
	functions},  Math. Nach. {\bf 236} (2002), 109--118.
\bibitem{FGM}
 I.C. Gohberg, A.S.  Markus and I.A.  Fel'dman,
 \textit{Normally solvable operators and ideals associated with them}, Bul. Akad. $\check{S}$tiince Moldoven \textbf{76} (10) (1960), 51--70 (English translation in \textit{Fourteen Papers on Functional Analysis and Differential Equations}, Am. Math. Transl. \textbf{61} (2), American Mathematical Society, Providence, RI, 1967, 63--84).
\bibitem{GG}
M. Gonz{\'a}lez and J.M. Guti{\'e}rrez, \textit{Polynomial Grothendieck
	properties,}  Glasgow Math. J. {\bf 37}  (2) (1995), 211--219.
\bibitem{Joaquin}
J.M. Guti{\'e}rrez,  \textit{Weakly continuous functions on Banach spaces not containing $\ell_1,$} Proc. Amer. Math. Soc. {\bf 119} (1) (1993), 147--152.
\bibitem{Iz1}
K. Is\'eki, \textit{On compact abelian semi-groups}, Michigan Math. J. \textbf{2} (1954), 59--60.
\bibitem{Iz2}
K. Is\'eki, \textit{On compact semi-groups},  Proc. Japan Acad.  \textbf{32} (1956), 221--224.

\bibitem{KPSTT}
A. Kami\'nska, A.I. Popov, E. Spinu, A. Tcaciuc and V.G. Troitsky, \textit{Norm closed operator ideals in Lorentz sequence spaces}, J. Math. Anal. Appl.  \textbf{389} (1) (2012), 247--260.
\bibitem{KL}
T. Kania and N.J. Laustsen, \textit{Uniqueness of the maximal ideal of the Banach algebra of bounded operators on $\mathcal{C}([0,\omega_1])$}, J. Funct. Anal. \textbf{262} (11) (2012), 4831--4850.
\bibitem{KK}
T.K. Kim and S. Krantz,  \textit{Normal families of holomorphic functions and mappings on a Banach space}, Expositiones Mathematicae \textbf{21} (2003),  193--218.
\bibitem{La}
N.J. Laustsen, \textit{Maximal ideals in the algebra of operators on certain Banach spaces}, Proc. Edinb.  Math. Soc. \textbf{45} (3)  (2002), 523--546.
\bibitem{LL}
N.J. Laustsen and R.J. Loy,  \textit{Closed ideals in the Banach algebra of operators on a Banach space.} In: \textit{Topological Algebras, their Applications, and Related Topics},  Banach Center Publ. \textbf{67} Polish Acad. Sci. Inst. Math., Warsaw, 2005, 245--264.
\bibitem{LLR}
N.J. Laustsen, R.J. Loy and C.J. Read, \textit{The lattice of closed ideals in the Banach algebra of operators on certain Banach spaces}, J. Funct. Anal. \textbf{214} (1) (2004), 106--131.
\bibitem{LSZ}
N.J. Laustsen, T. Schlumprecht and A. Zs\'ak, \textit{The lattice of closed ideals in the Banach algebra of operators on a certain dual Banach space}, J. Operator Theory  \textbf{56} (2) (2006), 391--402.
\bibitem{Mu}
J. Mujica, \textit{Complex Analysis in Banach Spaces},
 Dover
Books on Mathematics, 2010.
\bibitem{Pi}
A. Pietsch, \textit{Operator Ideals}, North-Holland Mathematical Library 20,  Amsterdam, 1980.
\bibitem{Ru}
W. Rudin, \textit{Functional Analysis}, McGraw-Hill, New York, 1991.
\bibitem{S}
A.L. Shields, \textit{On fixed points of commuting analytic functions,} Proc. Amer. Math. Soc.  \textbf{15} (1964), 703--706.
\bibitem{Su}
T.J. Suffridge, \textit{Common fixed points of commuting holomorphic  maps of the hyperball,} Michigan Math. J. \textbf{21}  (1974), 309--314.
\end{thebibliography}

\end{document}